\documentclass[10pt]{article}

\usepackage{amsmath,amsthm}
\usepackage{amssymb}
\usepackage{url}
\usepackage{srcltx}
\usepackage{enumerate}
\usepackage[latin1]{inputenc}
\usepackage{mathrsfs}
\usepackage{color}
\usepackage{curves}
\usepackage{eepic}
\usepackage{epic}
\usepackage{graphics,color}
\usepackage{pstricks}
\usepackage{pstricks-add,pst-node}
\usepackage{subfigure}
\usepackage{pstricks}
\usepackage{pstricks-add,pst-node}
\usepackage{curves}

\renewenvironment{proof}{\par \noindent \textbf{Proof.}
}{\hfill$\Box$\medskip}

\newtheorem{theorem}{Theorem}
\newtheorem{corollary}[theorem]{Corollary}

\newtheorem{proposition}[theorem]{Proposition}
\newtheorem{pr}[theorem]{Property}
\newtheorem{lem}[theorem]{Lemma}
\newtheorem{conj}[theorem]{Conjecture}

\date{\today}

 \title{Relaxed Locally Identifying Coloring of graphs }

  \author{M\'{e}ziane A\"{I}DER\thanks{Laboratoire LaROMaD.  Facult\'{e} des Math\'{e}matiques, U.S.T.H.B. 
   El Alia Bab-Ezzouar  16111, Algiers, Algeria}, Sylvain GRAVIER\thanks{Institut Fourier - SFR Maths \`{a} Modeler.UMR 5582 CNRS/Universit\'{e} Joseph Fourier \ 
    100 rue des maths, BP 74, 38402 St Martin d'H\`{e}res, France}, Souad~ SLIMANI \footnotemark[1], \footnotemark[2]}
  
\begin{document}

\maketitle

\begin{abstract}

A \textit{locally identifying coloring} ($lid$-coloring) of a graph is a  
proper coloring such that the sets of colors appearing in the closed neighborhoods of any pair of adjacent vertices 
having distinct neighborhoods are distinct. Our goal is to study a \textit{relaxed locally identifying coloring} 
($rlid$-coloring)  of a graph that is similar to locally identifying coloring for which the coloring is not  necessary 
proper.
We denote by $\chi_{rlid}(G)$ the minimum number of colors used in a relaxed locally identifying coloring of a graph $G$ 

In this paper, we prove that the problem of deciding that $\chi_{rlid}(G)=3$ for a $2$-degenerate planar graph $G$ is 
$NP$-complete. We give several bounds of  $\chi_{rlid}(G)$ and construct  graphs 
for which some of these bounds are tightened. 
Studying some families of graphs allows us to compare this parameter  with the minimum number of colors used 
in a locally identifying coloring of a graph $G$ ($\chi_{lid}(G)$), the size of a minimum identifying 
code of $G$ ($\gamma_{id}(G)$)  and the chromatic number of $G$ ($\chi(G)$).
\end{abstract}

\section{Introduction}\label{sec1}
Let $G=(V,E)$ be a simple undirected finite graph. Let $c: V\longrightarrow \mathbb{N}$  be a coloring of the vertices of $G$.
The coloring $c$  is  an \textit{identifying coloring} if  any pair of vertices $u$ and $v$ satisfies the property ($\mathcal P$): 
 $c(N[u])\neq c(N[v])$.  Observe that if $G$ has two distinct vertices $u$ and $v$ such that $N[u]=N[v]$ then there is 
no such coloring and we say the  vertices $u$ and $v$ are \textit{twins} in $G$.
Note that the coloring is not necessary proper. We define the 
\textit{identifying chromatic number} of $G$,
and denote by $\chi_{id}(G)$, the minimum number of colors required by an identifying coloring of $G$.
This notion was introduced by Parreau \cite{Aline}. In order to give a coloring version of the well-known 
\textit{identifying codes} defined by \textit{Karpovsky et al.} \cite{karpovsky}. 
 Parreau \cite{Aline} gave an upper and a lower bounds of the identifying chromatic number of a 
free-twin graph.  She characterized the free-twin graphs for which the identifying chromatic number  
is the number of vertices of the graph. 
\noindent In \cite{Parreau}, the notion of \textit{locally identifying coloring of graphs} ($lid$-coloring) 
was introduced and defined as follows: for any two adjacent vertices $u$ and $v$, the coloring $c$ satisfies both
the condition ($\mathcal{P}$) and the condition ($\mathcal{Q}$) : $c(u)\neq c(v)$.  
The  \textit{locally identifying chromatic number}
of $G$, denoted by $\chi_{lid}(G)$, is the smallest number of colors used in any $lid$ coloring of $G$.
Esperet et al. \cite{Parreau} gave several bounds of the  locally identifying chromatic number  for different families 
of graphs as planar graphs, some subclasses of perfect graphs and graphs with bounded maximum degree. They also proved 
that the problem to decide whether a subcubic bipartite graph with large girth is $3-lid$-colorable is an 
$NP$-Complete problem.
An upper bound for any graph was given in term of the maximum degree \cite{Parreau3} and the same authers also gave a bound 
of $\chi_{lid}$ for chordal graphs in term of the maximum degree and the chromatic number. 
\textit{Gon\c{c}alves et al.} \cite{Goncalves} proved that for any graph class of bounded expansion, the lid-chromatic number 
is bounded. They also showed that $\chi_{lid}$ is bounded for any class of minor closed classes of graphs and they gave 
an explicit upper 
bound for $\chi_{lid}$ of planar graphs.\\
\noindent Consider now a coloring of a graph $G$  satisfying only the condition ($\mathcal{P}$) for any pair of 
adjacent vertices.  We obtain the notion of  \textit{relaxed locally identifying coloring } ($rlid$-coloring) of a graph,on which we focus 
in this paper. Define the relaxed locally identifying chromatic number of a graph $G, \; \; \chi_{rlid}(G)$, as the smallest 
number of colors used in a relaxed locally identifying coloring.\\ 
Note that if $G$ contains twins $u$ and $v$ we have $c(N[u])=c(N[v])$.  One may ask which influence have twins for 
$rlid$-coloring?\\
To answer this question, let $\mathcal{R}$ be an equivalence relation defined as follows: for all vertices 
$u, \; \; v \in V(G)$, we have $u \mathcal{R}v $ if and only if $N[u]=N[v]$.
Denote by $G/ \mathcal{R}$ the maximal twin-free subgraph of $G$ and let $t(G)$ represent the number of  
equivalence-classes having at least two vertices in $G$.
\begin{theorem}{\label{premiere}}
Let $G/ \mathcal{R}$ be a maximal twin-free subgraph  of a connected graph $G$.
Then, we have $\chi_{rlid}(G/ \mathcal{R})-t(G)\leq \chi_{rlid}(G)\leq \chi_{rlid}(G/ \mathcal{R})$.
\end{theorem}
\begin{proof} 
Consider a $rlid$-coloring $c$ of $G/ \mathcal{R}$, and prove that $c$ also is a $rlid$-coloring of $G$.
For each vertex $x$ and its twin $y$ in $G$, put $c(x)=c(y)$. Since in $G$, we do not interest to distinguish the twins then 
$c$ defines a $rlid$-coloring of $G$.\\
Now, prove the second inequality. Let $c$ be a $rlid$-coloring of $G$.  Consider the coloring $c'$ defined as follows: 
$c'(u)=c(u)$ if the vertex $u$ has no twin in $G$ and color $t(G)$ other vertices of $G/ \mathcal{R}$  with different 
colors $\chi_{rlid}(G)+1$ until  $\chi_{rlid}(G)+t(G)$.  This coloring gives a $rlid$-coloring of $G/ \mathcal{R}$.
\end{proof}

Note that if $G$ is twin-free then $\chi_{rlid}(G)= \chi_{rlid}(G/ \mathcal{R})$.  In  Section \ref{sec4},  we exhibit 
an example for which the upper bound is tighten.
In this paper, we are interested in a studying $\chi_{rlid}$ of a graph $G$. If $G$ contains twins, we are interested  in 
separating
all pairs of adjacent vertices except twins in terms of $G/ \mathcal{R}$. 
We give several bounds of the relaxed locally identifying chromatic number
 for some subclasses of graphs and compare  $\chi_{rlid}$ with both $\chi_{lid}$  and $\chi$ (the chromatic number).\\
Our starting result is given as follows:
\begin{theorem}\label{main}
Let $G$ be a graph of order $n$ and $G/ \mathcal{R}$ be a maximal twin-free subgraph of $G$. \\
Then we have $\log \omega(G/ \mathcal{R}) +1 \leq \chi_{rlid}(G)\leq n$.
\end{theorem}
Where $\omega(G)$ represents the maximal size of a clique of $G$.\\
We will prove this theorem in  Section \ref{sec4}. In Section \ref{sec3}, we prove that the problem of 
deciding that $\chi_{rlid}(G)=3$ is $NP$-complete for a connected $2$-degenerate planar graph $G$ without twins and it is polynomial 
for a bipartite graph. In Section \ref{sec4}, we start by proving that the lower bound of Theorem \ref{premiere} is tight.
We characterize graphs $G$ satisfying $\chi_{rlid}(G)=n$ . We show the lower bound of Theorem \ref{main} and exhibit a 
family of graphs for which this bound is attained. We also study the split graphs for which we give an upper and a
lower bound of  $\chi_{rlid}$, and we construct two graphs which  tighten these bounds.\\
So, this paper is structured as follows: the next section presents basic definitions  
used in this paper. Then in Section $3$, we start by studying the complexity of this problem and we show that 
non trivial bipartite graphs are $3-rlid$. Further Section $4$ is spent to establish relationship between  $\chi_{id}$ , 
$\gamma_{id}$ and $\chi_{rlid}$. Section \ref{sec5} is spent to study the split graphs. 
We give an upper and a lower bound for these graphs and we exhibit a graph for which the lower bound is tight.
 Finally, we conclude by some remarks and some open questions.
\section{Useful definitions}\label{sec2}
Let $G=(V,E)$ be a finite connected graph, where $V$ (we also write $V(G)$) is the vertex set and $E$ (we also write $E(G)$ 
 is the edge set. 
We denote by $N(x)$ (resp. $N[x]$) the open (resp. closed) 
neighborhood of $x$, the set of all adjacent vertices of $x$ and  we have $N[x]=N(x) \cup \{x\}$.

A vertex $x$ is a twin of another vertex $y$ if we have  $N[x]=N[y]$. 
A graph $G$ is called \textit{twin-free} if $G$ contain no twin. 
The symmetric difference of two vertices $x$ and $y$, denoted by $N[x]\bigtriangleup N[y]$,
is the set of vertices  $(N[x]\setminus N[y])\cup (N[y]\setminus N[x])$. The maximal size of a clique in a graph $G$ 
is denoted by $\omega(G)$. 
A \textit{universal vertex} of a graph  is a vertex adjacent to all the others. 
\noindent A planar graph is a graph which can be drawn in the plane without any  edges crossing. 
Let $k$ be an integer, a graph $G$ is $k$-degenerate if every subgraph of $G$ has a vertex of degree at most $k$.\\
A subset $C$ of vertices of $G$ is \textit{an identifying code } of $G$ if $C$ is  a dominating set of $G$ 
(i.e. for each vertex $v\in V(G)$, we have  $N[v]\cap C \neq \emptyset$ ) and $C$ is a separating set of $G$ 
(i.e. for each pair of distinct vertices $u,v\in V(G)$, $N[u]\cap C \neq N[v]\cap C$). We denote by $\gamma_{id}(G)$ 
the minimum cardinality of an identifying code of $G$.\\
Since an $id$-coloring and a $lid$-coloring  are $rlid$-coloring, we have trivially \\
$\chi_{rlid}(G)\leq \chi_{id}(G)$ and  $\chi_{rlid}(G)\leq \chi_{lid}(G)$.\\
Given two graphs $G_1=(V_1,E_1)$ and $G_2=(V_1,E_1)$, ${G_1}\Join {G_2}$ is the \textit{join graph} of $G_1$ and $G_2$, in
which the vertex set is $V_1\cup V_2$ and the edge set is $E_1 \cup E_2 \cup \{v_1 v_2 |\;  v_1\in V_1, \; v_2\in V_2\}$.\\

Give $p\geq 2$ an integer, we define the graph $H_p=(V, E)$ from a clique $K$ of size $2^p$, 
and three stable sets $S_1, S_2, S_3$ of size $p$ as follows (See Figure. \ref{H_p}).\\
\begin{figure}[ht]
\begin{center}
\begin{pspicture}(0,0)(10,11.5)
\begin{tiny}
\psframe(1,1)(10,2.5)
\psellipse[fillcolor=lightgray](3,6)(2,0.75)
\psellipse[fillcolor=lightgray](8,6)(2,0.75)
\psellipse[fillcolor=lightgray](3,10)(2,0.75)

\pscircle[fillstyle=solid,fillcolor=black](1.3,2){0.07}
\pscircle[fillstyle=solid,fillcolor=black](1.7,2){0.07}
\pscircle[fillstyle=solid,fillcolor=black](2.2,2){0.07}
\pscircle[fillstyle=solid,fillcolor=black](4.2,2){0.07}
\pscircle[fillstyle=solid,fillcolor=black](4.7,2){0.07}
\psline[linestyle=dotted](2.6,2)(4,2)
\pscircle[fillstyle=solid,fillcolor=black](6.2,2){0.07}
\pscircle[fillstyle=solid,fillcolor=black](6.9,2){0.07}
\pscircle[fillstyle=solid,fillcolor=black](8.5,2){0.07}
\pscircle[fillstyle=solid,fillcolor=black](9.5,2){0.07}
\psline[linestyle=dotted](7.3,2)(8.7,2)

\pscircle[fillstyle=solid,fillcolor=black](1.7,6){0.07}
\pscircle[fillstyle=solid,fillcolor=black](2.2,6){0.07}
\pscircle[fillstyle=solid,fillcolor=black](4.2,6){0.07}
\pscircle[fillstyle=solid,fillcolor=black](4.7,6){0.07}
\psline[linestyle=dotted](2.6,6)(4,6)

\pscircle[fillstyle=solid,fillcolor=black](1.7,10){0.07}
\pscircle[fillstyle=solid,fillcolor=black](2.2,10){0.07}
\pscircle[fillstyle=solid,fillcolor=black](4.2,10){0.07}
\pscircle[fillstyle=solid,fillcolor=black](4.7,10){0.07}
\psline[linestyle=dotted](2.7,10)(4,10)

\pscircle[fillstyle=solid,fillcolor=black](6.2,6){0.07}
\pscircle[fillstyle=solid,fillcolor=black](6.9,6){0.07}
\pscircle[fillstyle=solid,fillcolor=black](9.5,6){0.07}
\pscircle[fillstyle=solid,fillcolor=black](8.5,6){0.07}
\psline[linestyle=dotted](7.2,6)(8,6)

\psline[linestyle=solid](1.7,2)(1.7,10)
\psline[linestyle=solid](2.2,2)(2.2,10)
\psline[linestyle=solid](4.2,2)(4.2,10)
\psline[linestyle=solid](4.7,10)(4.7,2)

\psline[linestyle=solid](6.2,2)(6.2,6)
\psline[linestyle=solid](6.2,2)(6.9,6)

\psline[linestyle=solid](6.9,2)(6.9,6)
\psline[linestyle=solid](6.9,2)(9.5,6)
\psline[linestyle=solid](6.9,2)(8.5,6)

\psline[linestyle=solid](9.5,2)(6.2,6)
\psline[linestyle=solid](9.5,2)(6.9,6)
\psline[linestyle=solid](9.5,2)(7.5,6)
\psline[linestyle=solid](9.5,2)(8.5,6)
\psline[linestyle=solid](9.5,2)(9.5,6)

\psline[linestyle=solid](8.5,2)(6.2,6)
\psline[linestyle=solid](8.5,2)(6.9,6)
\psline[linestyle=solid](8.5,2)(7.5,6)
\psline[linestyle=solid](8.5,2)(8,6)

\begin{tiny}
\uput[45](1,1.5){$x_\emptyset$}
\uput[45](1.5,1.5){$x_{\{1\}}$}
\uput[45](2,1.5){$x_{\{2\}}$}
\uput[45](3.7,1.5){$x_{\{p-1\}}$}
\uput[45](4.6,1.5){$x_{\{p\}}$}
\uput[45](9,1.5){$x_{\{2^{p} -1\}}$}
\uput[45](7.45,1.5){$x_Q$}
\end{tiny}\end{tiny}
\uput[45](10.1,1.6){$K$}
\uput[45](3,7){$S_1$}
\uput[45](3,11){$S_2$}
\uput[45](8,7){$S_3$}
\end{pspicture}
\caption{The graph $H_p$ and its components $K, S_1, S_2$ and $S_3$} \label{H_p}
\end{center}
\end{figure}
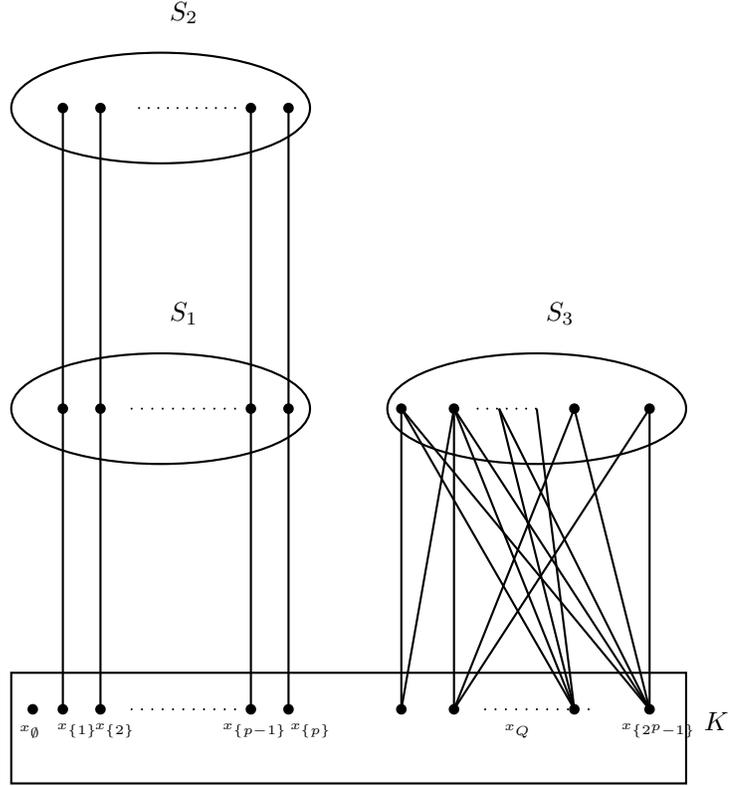

\noindent  $K=\{x_Q, Q\in \mathcal{P}(p)\}$: it means that $Q$ is a subset of $\{1,2,\dots,p\}$.\\
 $S_1=\{y_1, y_2, \dots, y_p\}$, $S_2=\{y'_1, y'_2, \dots, y'_p\}$ and 
$S_3=\{z_1, z_2, \dots, z_p\}$.\\
The edges are defined by:\\
$\circ$ for all $i\in [p]$,  $x_{\{i\}} y_i\in E$ and $y_i y'_i\in E$ .\\
$\circ$ and for all $Q \subseteq \mathcal{P}[p]$ with $|Q|\geqslant 2$, $x_{Q} z_i \in E$ iff $i\in Q$.\\
\noindent For an integer $k\geqslant 1$, we denote by $\textrm{P}^{k-1}_{2k}$ the $k-1$ power of the path on $2k$ vertices.

\noindent Given a graph $G$, we construct a graph $G^*$  by replacing each edge of $G$ by 
a path of length $3$, and adding to each vertex of $G$ a pendant vertex (See Figure. \ref{G^*}).
\begin{figure}[h]
\begin{center}
\begin{pspicture}(0,0)(10,4.5)
\pscircle[fillstyle=solid,fillcolor=black](1,1){0.07}\uput[45](0.7,0.5){$x$}
\pscircle[fillstyle=solid,fillcolor=black](3,1){0.07}\uput[45](3,0.5){$x'$}
\pscircle[fillstyle=solid,fillcolor=black](2,3){0.07}\uput[45](1.8,3.1){$y$}
\pscircle[fillstyle=solid,fillcolor=black](3,4){0.07}\uput[45](3.1,4){$y'$}
\psline[linestyle=solid](1,1)(3,1)
\psline[linestyle=solid](1,1)(2,3)
\psline[linestyle=solid](3,1)(2,3)
\psline[linestyle=solid](2,3)(3,4)

\pscircle[fillstyle=solid,fillcolor=black](6,1){0.07}\uput[45](5.7,0.5){$x$}
\pscircle[fillstyle=solid,fillcolor=black](8,1){0.07}\uput[45](8,0.5){$x'$}
\pscircle[fillstyle=solid,fillcolor=black](7,3){0.07}\uput[45](6.8,3.1){$y$}
\pscircle[fillstyle=solid,fillcolor=black](8,4){0.07}\uput[45](8.1,4){$y'$}
\pscircle[fillstyle=solid,fillcolor=black](6.6,1){0.07}
\pscircle[fillstyle=solid,fillcolor=black](7.4,1){0.07}
\pscircle[fillstyle=solid,fillcolor=black](6.3,1.6){0.07}\uput[45](5.8,1.4){$a$}
\pscircle[fillstyle=solid,fillcolor=black](6.7,2.4){0.07}\uput[45](6.3,2.2){$b$}
\pscircle[fillstyle=solid,fillcolor=black](7.3,2.4){0.07}
\pscircle[fillstyle=solid,fillcolor=black](7.7,1.6){0.07}
\pscircle[fillstyle=solid,fillcolor=black](7.3,3.3){0.07}
\pscircle[fillstyle=solid,fillcolor=black](7.7,3.7){0.07}

\psline[linestyle=solid](6,1)(8,1)
\psline[linestyle=solid](6,1)(7,3)
\psline[linestyle=solid](8,1)(7,3)
\psline[linestyle=solid](7,3)(8,4)

\pscircle[fillstyle=solid,fillcolor=black](9,4){0.07}
\pscircle[fillstyle=solid,fillcolor=black](6,3){0.07}
\pscircle[fillstyle=solid,fillcolor=black](5,1){0.07}\uput[45](4.8,0.5){$t$}
\pscircle[fillstyle=solid,fillcolor=black](9,1){0.07}

\psline[linestyle=solid](5,1)(6,1)
\psline[linestyle=solid](8,1)(9,1)
\psline[linestyle=solid](7,3)(6,3)
\psline[linestyle=solid](8,4)(9,4)
\uput[45](1.8,0.3){$G$}
\uput[45](6.8,0.3){$G^*$}
\end{pspicture}
\caption{The graph $G^*$ associated to the graph G where each edge in $G$ is replaced  by a path of length $3$ in $G^*$
 and a pendant vertex is attached to each vertex of $G$.}\label{G^*}
\end{center} 
\end{figure}
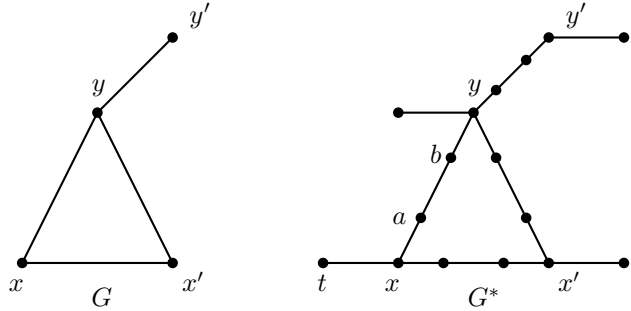
\section{Complexity results}\label{sec3}
This section is devoted to study the problem of complexity. Before we need to prove the following result which gives 
a necessary and sufficient condition for the graph $G^*$ (see Figure. \ref{G^*}) to have a $k-rlid$-coloring:
\begin{theorem} {\label{theoremG^*}} 
Let $k\geq 3$ be an integer and  $G=(V,E)$ be a connected twin-free graph.
Then  $G$ is $k$-colorable if and only if $G^{*}$ is $k$-rlid-colorable.
\end{theorem}
\begin{proof}
\noindent  Let $k\geq 3$ be an integer and $G$ be a connected twin-free graph of order $n\geq 3$.
First, prove that if $G$ admits a $k$-coloring then $G^{*}$ is  $k$-rlid-colorable.\\ 
Let $c$ be a $k$-coloring of $G$ and  $G^*$ be the graph associated to $G$.\\
Define a coloring $c'$  of $G^*$ from $c$ as follows:\\
- $c'(x)=c(x)$ for any vertex $x\in V(G)$;\\
- for a pendant vertex $t$ adjacent to $x$, $c'(t)=c(y)$ where $y$ is any vertex of $G$  adjacent to $x$.\\
- for $xy\in E(G)$, let $xaby$ be the corresponding path in $G^*$. Since $k\geq 3$, choose a color $q\neq c(x)$ and 
$q\neq c(y)$. Fix $c'(a)=c'(b) =q$.\\ 
\noindent We claim that $c'$ is a $rlid$-coloring of $G^*$.\\
According to the previous notation, we only have to check edges $xt$ or $xa$ or $ab$. 
Remark that for $x\in G, \; \; y\in G$ with $xy\in E(G), \; \; xaby$ is the corresponding path in $G^*$ and $c'(t)=c(y)$.
We have $c'(N[x])$ contains $\{c(y),c(x),c'(a)\}$, therefore $|c'(N[x])|\geq 3$. Since $t$ is a leaf, we have $|c'(N[t])|\leq 2$.
Thereover by definition of $c'$, for any subdivision $x'a'b'y'$ we have $c'(a')=c'(b')$ then $|c'(N[a'])|\leq 2$.
Thus extremities of $xt$ or $xa$ edges are identified. 
In any path $xaby, \; \; c'(N[a])=\{c'(a), c(x)\}, \; \; c'(N[b])=\{c'(a), c(y)\}$. Since $c(x)\neq c(y)$ by definition 
of a coloring $c$ in $G$ and since $xy\in E(G)$, we obtain that $c'(N[a])\neq c'(N[b])$. 

Let $c$ be a $k-rlid$-coloring of $G^*$. To achieve the proof it is enough to show that $c(x)\neq c(y)$  
for any edge $xy\in E(G)$.
If $c(x)\neq c(y)$  then let $xaby$ be the  path of $G^*$ corresponding to edge $xy$ in $G$.
If  $c(x)= c(y)$  then $c(N[a])=\{c(a), c(x),c(b)\}=c(N[b])$, a contradiction.
\end{proof}

For  planar graphs, we have the following result\\

\begin{theorem}\label{complexite} \cite{Lichte}
 The problem of $3$-colorability of a planar graph is $NP$-complete.
\end{theorem}

\noindent Let $G$ be a planar graph, and  $G^*$ be the graph associated to $G$. Then $G^*$ is also planar. 
On the other hand, $G^*$ is $2$-degenerate (by construction).
By Theorem \ref{theoremG^*}, $G$ is $3$-colorable iff $G^*$ is $3-rlid$-coloring. By Theorem \ref{complexite}, 
the problem of deciding if a  planar graph $G$ is $3$-colorable is $NP$-complete, then we deduce that the problem 
of deciding that $G^*$ is $3-rlid$-coloring is also $NP$-complet. \\
Then, we obtain the following result: 

\begin{corollary}\label{2degenerate}
 The problem of deciding that a $2$-degenerate planar graph is $3-rlid$-coloring  is $NP$-Complete.\qed
\end{corollary} 

In \cite{Parreau}, it was shown that:
\begin{theorem}\cite{Parreau}
For any fixed integer $g$, deciding whether a bipartite graph with girth at least $g$ and maximum degree $3$ 
is $3-lid$-colorable is an $NP$-complete problem.
\end{theorem}

Thanks to Corollary \ref{2degenerate}, $rlid$-coloring is still $NP$-complete for $2$-degenerate planar graphs.  
However, we will show that it is polynomially solvable for bipartite graphs.

\begin{theorem}{\label{bipartite}}
 Let $G$ be a bipartite graph of order at least $3$, then $\chi_{rlid}(G)\leq 3$.
\end{theorem}
\begin{proof}
\noindent Let $G$ be a bipartite graph.
We have two cases: $G$ contains a universal vertex or no.\\
- \textbf{$G$ contains a universal vertex}.\\
Then $G\backsimeq K_{1,p}$  with $p>1$ then: \\
$\bullet$ If $p=2$. Since $K_{1,2} \simeq P_3$, one may assume that  $\chi_{rlid} (K_{1,2})\leq 3$.\\
$\bullet$ If $p>2$: Let us prove that $\chi_{rlid} (K_{1,p}) \leq 3$.
Assume that $u$ is the universal vertex and the other vertices are denoted by  $\{v_1, v_2, \dots, v_p\}$.
The coloring $c$ defined by $c(u)=1, \; \; c(v_1)=2$ and $c(v_i)=3$ for all $i=2, \dots, p$ is a $rlid$-coloring.\\
- \textbf{$G$ does not contain an universal vertex}. \\
Let $x$ be a vertex of $G$.
Considering the partition of vertices of $G$ into levels $L_0, L_1,...L_p$ according to the vertex $x$, 
we have $L_0=\{x\}$ and $L_j= \{v/ d(x,v)=j\}$ for $j\geqslant 1$. 
Since $G$ is a finite graph, we have a finite number of levels.
Let $A_i$ and $B_i$ be a subdivision of $L_i$ in two disjoint subsets such that $v\in A_i$ iff $N(v)\cap L_{i+1} =\emptyset$ 
and $N(v)\cap L_{i-1}=B_{i-1}$, and $v\in B_i$ iff  $N(v)\cap L_{i+1} =L_{i+1}$ and 
and $N(v)\cap L_{i-1}=B_{i-1}$.\\
Let $v$ be a vertex of $G$ ($v\neq x$). Consider the coloring $c$ of $G$ with three colors $\{1,2,3\}$ as follows, for all $i\geqslant 0$:\\
\begin{center}
$c(v)=$
$\begin{cases}
{1} &  if \; \;  v\in L_{4i} \\
{1} &  if  \; \; v\in A_{4i+1}\\
{2} &  if  \; \; v\in  B_{4i+1}\\
{3} &  if \; \; v\in L_{4i+2} \\
{3} &  if  \; \; v\in  A_{4i+3}\\
{2} &  if  \; \; v\in B_{4i+3}
\end{cases}$
\end{center}
Since $G$ is a bipartite graph, then we can not have edges between vertices in a same level. Observe that $L_2$ is not 
empty.\\
For all $i\geqslant 0$, remark that for $v\in L_{4i}, \; \; c(N[v])=\{1,2\}$.\\
When $v\in A_{4i+1}$, we get $c(N[v])=\{1\}$, if $v\in B_{4i+1}$ we have $c(N[v])=\{1,2,3\}$.\\
If $v\in L_{4i+2}$, we obtain $c(N[v])=\{2,3\}$.
Finally in $L_{4i+3}$, if $v\in A_{4i+3}$ we get $c(N[v])=\{3\}$, if $v\in B_{4i+3}$ we obtain $c(N[v])=\{1,2,3\}$.\\
Let $u$ and $v$ be two adjacent vertices. Since $G$ is a bipartite graph, then  $u$ and $v$ do not belong to a same level.
Suppose that $u\in L_i$, then we have either $v\in L_{i-1}$ or $v\in L_{i+1}$ and suppose that $N[u]\neq N[v]$.
We have to consider four cases:\\

\noindent  $\bullet$ \textbf{Case $1. \; \; u\in L_{4i}$}\\
If $u\in L_{4i}$ and  $v\in B_{4i-1}$, we have $c(N[u])\bigtriangleup c(N[v])=\{3\}$.\\
If $u\in B_{4i}$, then $v\in L_{4i+1}$ and $c(N[u])\bigtriangleup c(N[v])=\{2 \; \; or\; \; 3\}$.\\
$\bullet$ \textbf{Case $2. \; \; u\in L_{4i+1}$}\\
If $u\in A_{4i+1}$, if $v\in B_{4i}$ we have $c(N[u])\bigtriangleup c(N[v])=\{2\}$.\\
If $u\in B_{4i+1}$, then either $v\in B_{4i}$ and $c(N[u])\bigtriangleup c(N[v])=\{3\}$ 
or $v\in L_{4i+2}$ and we get then $c(N[u])\bigtriangleup c(N[v])=\{1\}$ \\
$\bullet$ \textbf{Case $3. \; \; u\in L_{4i+2}$}\\
If $u\in L_{4i+2}$ and  $v\in B_{4i+1}$, we obtain  $c(N[u])\bigtriangleup c(N[v])=\{1\}$.\\
If $u\in B_{4i+2}$, then either $v\in A_{4i+3}$ and $c(N[u])\bigtriangleup c(N[v])=\{2\}$
or $v\in B_{4i+3}$ and $c(N[u])\bigtriangleup c(N[v])=\{1\}$.\\
$\bullet$ \textbf{Case $4. \; \; u\in L_{4i+3}$}\\
If $u\in A_{4i+3}$, if we have $v\in B_{4i+2}$ then $c(N[u])\bigtriangleup c(N[v])=\{2\}$.\\
If $u\in B_{4i+3}$, then either $v\in B_{4i+2}$ and $c(N[u])\bigtriangleup c(N[v])=\{1\}$
or $v\in L_{4(i+1)}$ and $c(N[u])\bigtriangleup c(N[v])=\{3\}$.\\
In all case, observe that $u$ and $v$ are distinguished. Then we have  $\chi_{rlid}(G)\leq 3$.
\end{proof}\\

Remark that there is no graph $G$  with $\chi_{rlid}(G)=2$. Moreover, a graph is $1-rlid$-coloring iff it is 
the disjoint union of cliques. Therefore, our proof of Theorem \ref{bipartite} provides a polynomial time algorithm 
for $3-rlid$-coloring graph if it is bipartite.

\section{Relationship between $\gamma_{id}, \; \; \chi_{id}$ and $\chi_{rlid}$} \label{sec4}

The notion of \textit{identifying chromatic number} $\chi_{id}$ and \textit{locally identifying chromatic number} $\chi_{lid}$ were 
given for twin-free graphs, by as against for the \textit{relaxed locally identifying chromatic number}, 
 we even study  the graphs 
that contain twins. In the following , we show  that the lower bound of inequality of Theorem \ref{premiere} is tight.
\begin{pr}
 Let  $p\geq 4$ be an integer and $t=\binom{p-1}{2}$
There is a graph $G$ such that $\chi_{rlid}(G)= p$ 
and  $\chi_{rlid}(G)=\chi_{rlid}(G/ \mathcal{R})-t$.
\end{pr}
\begin{proof}\\
Construct the graph $G$ such that $\chi_{rlid}(G)= p$ and $t=\binom{p-1}{2}$. Consider $K_{p+t}$ with 
$\{x_1, x_2, \dots,x_t, \dots, x_{p+t}\}$   the set of its vertices,   ${K^*_{p+t}}$  is the graph associated to 
 $K_{p+t}$ as defined in Section \ref{sec2}. Denote by   $z_i$  the pendant vertex of $x_i$ for $i=1,\dots,  p+t$ and
  $x_{i}^{j}$  and $x_{j}^{i}$ are the vertices subdivide the edge $x_ix_j$  with  $x_{i}^{j}$ (resp.  $x_{j}^{i}$) 
is adjacent to $x_i$ (resp. to $x_j$) for $i,j =1, \dots, p+t$.
The graph $G$ is obtained from $K^*_{p+t}$ by adding $t$ twins $y_{1}, y_{2}, \dots, y_{t}$ respectively 
to $x_{1}, x_{2}, \dots, x_{t}$. \\
First, observe that $G / \mathcal{R} \backsimeq K^*_{p+t}$. 
Then $\chi_{rlid}(G/ \mathcal{R}) =p+t$.\\
Now, we will prove that $G$ admits a $p-rlid$-coloring.\\
Let $c$ be a coloring of $G$ defined as follows: 
put $c(x_{t+i})=i$ for $i\geq 1$.  For the vertices belonging to the same equivalence-class, $c$ is defined by 
one to one  mapping $\{(c(x_{i}), c(y_{i})) \mid \;  1\leq i\leq t\}\longmapsto \mathcal P{_2}(p-1)$.  
The vertices $x_{i}^{j}$ and $x_{j}^{i}$ receive the same color $p$  for all $(i,j)$ except  $(p, t+1)$ and 
we put $c(x_{p}^{t+1})=c(x_{t+1}^{p})=p-1$.  
For the leaf $z_i$ adjacent to $x_i$, 
put $c(z_i)=1$ if $i>t$ and $c(x_i)\neq 1$ and if $c(x_i)=1,\; \;  z_i$ receives the color $2$. 
Put $c(z_i) \in \{1, \dots, p-1\}\setminus \{c(x_i), c(y_i)\}$ 
if $i\leq t$.\\
Let $u$ and $v$ be two adjacent vertices  of $G$ (we are not interested to distinguish two vertices belonging to 
a same equivalence-class).\\
Observe that if  $u=x_i^j$ and $v=x_j^i$ with $1\leq i, j\leq t$, we have $c(z_i)\in c(N[x_i^j])$ and 
$c(z_i)\notin c(N[x_j^i])$.
If  $t+1\leq i, j\leq p+t$, we have $c(x_i)=i\in c(N[x_i^j]), \; \; c(x_j)=j\notin c(N[x_i^j])$ and $c(x_i)\neq c(x_j)$.
If $1\leq i\leq t$ and $t+1\leq j\leq p$, remark that $\mid c(N[x_j^i])\mid = 3$ and  $\mid c(N[x_i^j])\mid = 2$.
If $u=x_i$ and $v=x_i^j$, then $c(z_i)\in c(N[x_i])$ with $c(z_i)\notin c(N[x_i^j])$ if $1\leq i\leq t$  and   
$\mid c(N[x_j^i])\mid = 2$ 
and     $\mid c(N[x_i])\mid = 4$ if $t+1\leq i\leq p+t$.\\
If $u=x_i$ and $v=z_i$, then if $1\leq i\leq t$, we get  $\mid c(N[x_i])\mid = 4$ and  $\mid c(N[z_i])\mid = 3$.
If    $t+1\leq i\leq t+p$, we have   $\mid c(N[x_i])\mid = 3$ and $\mid c(N[z_i])\mid = 2$.\\
Then, each two adjacent vertices of $G$ are distinguished by $c$.
\end{proof}\\

Theorem \ref{bipartite} claims that bipartite graphs admit $3-rlid$-coloring despite $\chi_{id}$ of bipartite graph 
is not bounded. Howeover in \cite{Parreau}, it is shown the following result:
\begin{theorem}\label{parreau}\cite{Parreau}
Let $G$ be a free-twin graph. Then $\chi_{id} (G) \leq \gamma_{id} (G) +1$.
\end{theorem}
For $rlid$-coloring, this gives:
\begin{pr}\label{proper1}
For any twin-free graph $G$, we have $\chi_{rlid} (G)  \leq \gamma_{id} (G) +1$.\qed
\end{pr}

Our goal is to use this inequality to characterize graphs $G$  satisfying $\chi_{rlid}(G)=n$.
A characterization of  graphs for which the $id$-chromatic number equals the order of $G$ is
given in \cite{Aline}.
\begin{theorem}\cite{Aline}\label{chi(id)}
Given a connected twin-free graph $G$, we have $\chi_{id}(G)=n$ if and only if $G$ is a complete graph minus 
maximal matching or $G= K_1 \Join {\mathcal{H}}$ where ${\mathcal H}={\mathcal H {_1}}\bowtie \dots \bowtie{\mathcal H {_l}}$
 with ${\mathcal H_{i}} \backsimeq {\textrm{P}^{k-1}_{2k}}$  or ${\mathcal H{_i}}= \overline{K_2}$ 
for $i=1, \dots,l, \; \; k\geq 2$.
\end{theorem}
For a relaxed locally identifying coloring we obtain:
\begin{corollary}\label{caract}
 Let $G$ be a  connected twin-free graph of order $n$. Then,\\
 $\chi_{rlid} (G)=n$ if and only if  $G\simeq K_1 \Join {\mathcal H}$ with 
${\mathcal H}= {\mathcal H{_1}}\Join {\mathcal H{_2}}\Join \dots \Join {\mathcal H{_l}}$  and 
${\mathcal H{_i}}\backsimeq {\textrm{P}^{k-1}_{2k}}$  or ${\mathcal H{_i}}= \overline{K_2}$ for $i=1, \dots, l$.
\end{corollary}
\begin{proof}
 Let $G$ be a connected twin-free graph of order $n\geq 3$.\\
\underline{The proof of \textbf{if part}}.\\
\textbf{If $G\simeq K_1 \Join \mathcal H$} where  
${\mathcal H}= {\mathcal H{_1}}\Join {\mathcal H{_2}}\Join \dots \Join {\mathcal H{_l}}$
 with $\mathcal H{_i}\simeq {\textrm{P}^{k-1}_{2k}}$ or ${\mathcal H{_i}}= \overline{K_2}$ 
and $i=1, \dots, l$,  
let $u$ be the universal vertex corresponding to $K_1$ and  $\{v_1,v_2,\dots,v_{2k}\}$ the set of vertices of 
$\mathcal H{_1}$ for which $\mathcal H{_i}\backsimeq \textrm{P}^{k-1}_{2p}$. Let $c$ be an $rlid$-coloring of $G$.\\
For all $s$ such that $2k>s>k$, we have $N[v_s]\bigtriangleup N[v_{s+1}]=\{v_{s-k}\}$. 
Hence for all $j,t\leq k, \; \; c(v_j)\neq c(v_t)$. Moreover, for all $s>k$, we have  
$N[v_s]\supseteq \{v_{k+1}, \dots, v_{2k},u\}\cup \displaystyle \bigcup_{t\geq 2} V(\mathcal H{_t})$.\\
Hence $\forall j\leq k$, 
we have $c(v_j)\notin c(\{v_{k+1}, \dots, v_{2k}, u\} \cup \displaystyle \bigcup_{t\geq 2} V(\mathcal H{_t}))$. 
Otherwise, $c(N[v_{j+k-1}])=c(N[v_{j+k}])$.\\
By symmetry, for any $j>k$,  we have  $c(v_j)\notin c(V\setminus v_j)$.\\
Similarly, we obtain the same results for each $\mathcal H{_j}$ with $j\geq 2$.\\
\textbf{If ${\mathcal H{_1}} \simeq \overline{K_2}$}, let $a$ and $b$ be the vertices corresponding to $\overline{K_2}$.
We get $\{b\}=N[a]\bigtriangleup N[u]$. Moreover $N[a]= V\setminus \{b\}$, then $b$ has a color different  
to all colors of other vertices of  graph. We obtain a similar result for the vertex $a$.\\
Finally all vertices of $G$ have different colors. \\

\underline{The proof of \textbf{only if part}}.\\
If $\chi_{rlid}(G)=n$ then $\chi_{rlid}$-coloring is $\chi_{id}$-coloring and by Theorem \ref{chi(id)} \cite{Aline} 
we have the result.
 \end{proof}\\
 
Corollary \ref{caract} shows that $\chi_{rlid}(G)=\chi_{id}(G)$ when $\chi_{rlid}(G)=n$ 
then  we also have $\chi_{rlid}(G)= \chi_{lid}(G)$.\\
Now, we can show that $\chi_{rlid}$ may be very small compared with $\chi_{lid}$ , for this we will show the lower bound of 
Theorem \ref{main}. We will exhibit a family of graph for which this bound is tight.  
\begin{lem}\label{important}
Let $G$ be a graph and $G/ \mathcal{R}$ be the maximal free-twin subgraph of $G$. Then $\chi_{rlid}(G)\geq \log(\omega(G/ \mathcal{R}))+1$.
\end{lem}
\begin{proof}
Let $K$ an $\omega$-clique of $G/ \mathcal{R}$. Let  $V'=V(G)\setminus K$ and  $c$ be a $rlid$-coloring of  $G$.\\
Since for each pair of distinct adjacent vertices  $x, y$ in $K$, we must have $c(N[x])\neq c(N[y])$ 
(because $c$ is  an $rlid$-coloring), 
we also have for each vertex $x\in K$, $c(K)$ belongs to $c(N[x])$, since $x$ has in its  neighborhoods all vertices of $K$
 and some vertices of $V'$.\\
Thus, for each two distinct adjacent vertices $x, y \in K$,\\ we have $c(N[x])\neq c(N[y])$  iff 
$c(N[x])-c(x)\neq c(N[y])-c(y)$.

Let  $p=|c(V)\setminus (c(K)|$. With $p$ distinct colors, we can distinguish at most $2^p$ vertices.
Thus  $2^p\geq |K|$. Hence $p\geq \log(|K|)$. Then, the total number of different colors used by $c$ is at least $|c(K)|+p\geq p+1$.
 \end{proof}

Now we exhibit a family of graphs satisfying  the equality:
\begin{corollary}\label{Hp}
If $p\geqslant 2$, then $\chi_{rlid}(H_p)= \log \omega(H_p)+1$.
\end{corollary}
\begin{proof}
 We define a $rlid$-coloring $c$ as follows:\\
$\bullet$ $c(x_Q)=p+1$ for all $Q\subseteq \mathcal P([p])$\\
$\bullet$ $c(y_i)=c(z_i)=i$ for all $i\in [p]$.\\
$\bullet$ $c(y'_i)=i+1 p\; \; mod\; \; p$ for all $i\in [p]$.\\
Now, we have $c(N[x_Q])=Q\cup \{p+1\}$ for all $Q\subseteq \mathcal P([p])$, $c(N[z_i])=\{i, p+1\}$ for each $i \in[p]$,
$c(N[y'_i])=\{i, i+1\; \;  mod \; \; p\}$ for each $i$ and  $c(N[y_i])=\{i, i+1 \; \;  mod \; \; p, p+1\}$ for each $i$.
Since there is no edge between any $y\in S_1$ and any $x_Q \in K$ with $|Q|\geqslant 2$ and no edge between any $z\in S_3$ 
and any $i_{\{i\}}$ for any 
$i \in [p]$, then $c$ defines a $rlid$-coloring.  
\end{proof}\\

From Corollary \ref{Hp}, we have $\chi_{rlid}(H_p) =p+1$  despite 
$\chi_{lid}(H_p) \geq \chi (H_p)\geq \omega (H_p)= 2^{p}$.

\section{Split graphs}\label{sec5}
\noindent Let $G=(K\cup S,  E)$ be a split graph with $K=\{x_1, x_2,...,x_k\}$ a clique of maximal size $\omega(G) =k$ and 
$S=\{s_1, s_2,..., s_p\}$ a stable of $G$. 
For this class of graphs, the $\chi_{rlid}$ is given by the following theorem :
\begin{theorem}\label{Split}
 Let $G=(K\cup S,E)$ be a connected split graph without twins.\\
Then,  $\log (\omega(G)) +2 \leq \chi_{rlid} (G)\leq \omega(G)+2$.
\end{theorem}
\begin{proof}
\noindent Let us prove that $\chi_{rlid}(G) \geq \log(\omega(G))+2$. \\
In order to identify all vertices of $K$, we have $\omega(G) \leq 2^p$. Let $c$ be a $rlid$-coloring of $G$.\\ 
According to the proof of Theorem \ref{important}, we have $\mid c(V)\setminus c(K) \mid\geq p$. Let us prove that we 
need to add at least two colors to color $K$.\\
If $|c(K)|\geq 2$, then we are done.\\
If $\omega < 2^p$, then we are done.\\
Therefore, we can assume that $\omega(G) =2^p$ and $|c(K)|=1$.
For any vertex $x\in K$ put  $c(x)=p+1$.
Let $u\in S$ and denote by $I(u)$ the set of colors appearing in $N[u]$. 
Observe that $I(u)$ contains the color $p+1$. 
If $u$ is adjacent to some vertices in $K$. Since $\omega(G) =2^p$ then there exists a vertex $x\in K$ such that 
$c(N[x])=\{i,p+1\}$ 
with $i=c(v)$ and $v\in S$. This shows that $v$ is the vertex belonging to $S$, adjacent to $x$ which gives $c(N[x])=c(N[v])$,
 a contradiction. Then in this case, we need to add at least two colors in $K$ and we are done.

Now, let us prove $\chi_{rlid}(G)\leq \omega(G)+2$. 
One may assume that $G$ is a connected graph and $K$ is a maximal clique of $G$ of size $\omega$.
Start by proving that $k-1$ colors in $S$ suffice to distinguish all vertices in $K$. 
Inductively, we construct a subset $S'$ of $S$ with $|S'|\leq k-1$ (where $k=|K|$) such that for any  
two distinct vertices $x$ and $y$ in $K$, we have $N[x]\cap S' \neq N[y]\cap S'$.\\
Choose any vertex $u\in S$. Since $K$ is maximal then $K\cap N[u]\neq \emptyset$.
Denote $K_1=N[u]\cap K$ and  $K_2= K\setminus K_1$. Let $G_i$ ($i=1,2$) be the split graph induced by $K_i$ and $S_i$ 
defined by $v\in S_i$ iff  $v\in S$ and $N[v]\subseteq K_i$.
Let us prove that $G_i$, $i=1,2$, is twin-free and that $K_i$ ($i=1,2$) is maximal.\\
By construction,  $K_i$ ($i=1,2$) is maximal. \\
Now suppose that $x, y\in G_i$ are twins. Since $K_i$ is maximal then 
$x, y \in K_i$. By definition of $G_i$, any vertex $x_i\in V(G)\setminus V(G_i)$ is either adjacent or not to  $x$ and $y$. 
Therefore, $x$ and $y$ will be twins in $G$, which yields a contradiction.
Now, apply the induced hypothesis in $G_1$ and $G_2$. We get $|S'|\leq k-1$.\\
So set $S'=S_1\cup S_2$  and let $c(s_i)=i$ for any vertex $s_i\in S'$.
Since $N[x]\cap S' \neq N[y]\cap S'$ for all two distinct vertices $x,y\in K$ then there is at most one vertex namely 
$u\in K$ such that $N[u] \cap S' =\emptyset$.
Now, partition $K$ in $K_1$ and $K_2$ such that any vertex $x\in K_1$ satisfies $|N[x]\cap S'|=1$ and any vertex $y\in K_2$ 
satisfies  $|N[x]\cap S'|\geq 2$.
Let $c$ be a $rlid$-coloring of $G$.

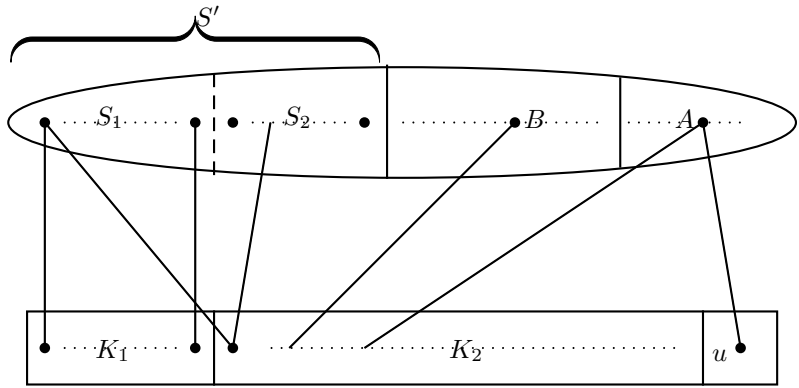
\begin{figure}[h]
\begin{center}
\begin{pspicture}(0,0)(11,7)
\psellipse[fillcolor=lightgray](5.5,4.5)(5.25,0.75)
\psframe[fillcolor=lightgray](0.5,1)(10.5,2)

\pscircle[fillstyle=solid,fillcolor=black](0.75,1.5){0.07}
\pscircle[fillstyle=solid,fillcolor=black](2.75,1.5){0.07}
\pscircle[fillstyle=solid,fillcolor=black](3.25,1.5){0.07}
\pscircle[fillstyle=solid,fillcolor=black](10,1.5){0.07}
\psline[linestyle=solid](3,1)(3,2)
\psline[linestyle=solid](9.5,1)(9.5,2)

\psline[linestyle=dotted](1,1.5)(2.5,1.5)
\psline[linestyle=dotted](3.75,1.5)(9.25,1.5)

\pscircle[fillstyle=solid,fillcolor=black](0.75,4.5){0.07}
\pscircle[fillstyle=solid,fillcolor=black](2.75,4.5){0.07}
\pscircle[fillstyle=solid,fillcolor=black](3.25,4.5){0.07}
\pscircle[fillstyle=solid,fillcolor=black](5,4.5){0.07}
\pscircle[fillstyle=solid,fillcolor=black](7,4.5){0.07}
\pscircle[fillstyle=solid,fillcolor=black](9.5,4.5){0.07}
\psline[linestyle=solid](5.3,3.75)(5.3,5.27)
\psline[linestyle=solid](8.4,3.9)(8.4,5.1)

\psline[linestyle=dotted](1,4.5)(2.5,4.5)
\psline[linestyle=dotted](3.5,4.5)(4.75,4.5)
\psline[linestyle=dotted](5.5,4.5)(8.1,4.5)
\psline[linestyle=dotted](8.5,4.5)(10,4.5)
\psline[linestyle=dashed](3,3.8)(3,5.15)
\uput[45](1.3,4.3){$S_1$}

\uput[45](3.8,4.3){$S_2$}
\psbrace*[linecolor=black,ref=lf,rot=-90](5.2,5.3)(0.3,5.3){$S'$}
\uput[45](7,4.3){$B$}
\uput[45](9,4.3){$A$}
\uput[45](1.3,1.2){$K_1$}
\uput[45](6,1.2){$K_2$}
\uput[45](9.5,1.2){$u$}
\psline[fillstyle=solid,fillcolor=black](0.75,1.5)(0.75,4.5)
\psline[fillstyle=solid,fillcolor=black](2.75,1.5)(2.75,4.5)
\psline[fillstyle=solid,fillcolor=black](3.25,1.5)(0.75,4.5)
\psline[fillstyle=solid,fillcolor=black](3.25,1.5)(3.75,4.5)
\psline[fillstyle=solid,fillcolor=black](4,1.5)(7,4.5)
\psline[fillstyle=solid,fillcolor=black](10,1.5)(9.5,4.5)
\psline[fillstyle=solid,fillcolor=black](9.5,4.5)(5,1.5)

\end{pspicture}
\caption{The split graph with its components $K$ and  $S$.}
\label{split}
\end{center}
\end{figure}
 
\textbf{\underline{Case 1} : there exists $u\in K$ such that $N[u]\cap S\neq \emptyset$}. \\
Let $A=N[u]\cap S$ and $B=S\setminus (S'\cup A)$ (see Figure.\ref{split}).\\
\textbf{\textit{Subcase 1.1}}. $K_1\neq \emptyset$ and $A= \emptyset$.
Consider the following coloring: assign to all vertices in $B\cup K_1\cup K_2$ the color $k$ and put $c(u)=k+1$.
Let $x\in K$ and $s\in S$ be two vertices in $G$. We know that all vertices of $K$ are distinguished by $S'$ and
since $k+1\in c(N[x])\setminus c(N[s])$ then $x$ and $s$ are distinguished.\\
\textbf{\textit{Subcase 1.2}}. $K_1\neq \emptyset$ and there exists $y\in K_1$ such that $N[y]\cap A\neq \emptyset$.
We define a coloring $c$ as follows: all vertices of $K_1\setminus \{y\}\cup K_2 \cup (S\setminus S')$ receive  the color $k$,
put $c(y)=k+1$ and $c(u)=k+2$.\\
We claim that any two adjacent vertices of $G$ are discriminated. Indeed, we know that all vertices of $K$ are distinguished
 by $S'$.
Now let us cheek the edges between $K$ and $S$. Let $x\in K$ and $s\in S$ be two adjacent vertices. Suppose that $N[x]=N[s]$.\\
$\bullet$ \textbf{If $x\in K_1$}, since $k+2\in c(N[x])$ then $s\in A$. Thus by assumption, $s$ and $x$ are not adjacent.\\
$\bullet$ \textbf{If $x\in K_2$}, since $k+1$ belongs to $c(N[x])$ then $s\in B\cup S'$. Thus $k+2$ belongs to 
$c(N[x])\cap c(N[s])$.\\
$\bullet$ \textbf{If $x=u$}, then $s\in A$. Thus $k+1\in c(N[x])\cap c(N[s])$, which yields a contradiction.\\
\textbf{\textit{Subcase 1.3}}. $K_1\neq \emptyset$ and for all $x\in K_1$ we have $N[x]\cap A\neq \emptyset$.\\
$\bullet$ \textbf{If $A=\{v\}$ }, 
then there exists $y\in K$ not adjacent to $v$.  Put $c(v)=k, \; \; c(u)=k+2, \; \; c(y)=k+1$ and 
 for all $z\in (K\setminus \{y\}) \cup B$, put $c(z)=k$.\\
If $x\in K$ and $s\in S\setminus A$,  since $k+2\in c(N[x])\setminus c(N[s])$  then $x$ and $s$ are distinguished .\\
Additionally, the pair $x\in K$ and $s=v$ are also distinguished since $k+1\in c(N[x])\setminus c(N[v])$.\\
$\bullet$ \textbf{If $|A|\geq 2$ },  
consider the following coloring: the vertices of $B\cap K$ receive the color $k$. There exists a vertex $w\in A$ such that 
$c(w)=k+1$ and $\forall v\in A\setminus \{w\}, \; \;c(v)=k+2$\\
Let $x\in K$ and $s\in S$ be two vertices. \\
- \textbf{If $s\notin S'$ and $x\neq u$ }, 
then the color $k+1\in c(N[x])\setminus c(N[s])$.\\
- \textbf{If $s\in S'\cup B$ and $x=u$},
the vertices $x$ and $s$ are not adjacent.\\
- \textbf{If $s\in A$ and $x=u$},
one of the colors $k+1$ or $k+2$ belongs to $c(N[s])$ but both are in $c(N[u])$.\\
- \textbf{If $s\in S'$ and $x\in K_2$},
there exists a vertex $s'\in S'\setminus \{s\}$ such that $c(s')\neq c(s)$. We have $c(s_j)\in c(N[x])\setminus c(N[s])$\\
- \textbf{ If $x\in K_1$ and $s\in S'$}, 
either the color $k+1$ or $k+2$ belongs to  $ (N[x])\setminus c(N[s])$.

\textbf{\underline{Case 2} : The vertex $u$ does not exist in $K$}. In this case $A=\emptyset$ and :\\
$\bullet$ If $K_1 =\emptyset$,
consider the following coloring: all vertices in $K\cup B$ receive the color $k$. \\
Let $x\in K$ and $s\in S$ be two vertices in $G$.Then: \\
- \textbf{If $s\in S'$}, since  $|N[x]\cap S'|\geq 2$ then there exists a vertex $s'\in S'$ 
such that $s\neq s'$ and $c(s')\in c(N[x])\setminus c(N[s])$.\\
- \textbf{If $s\in S\setminus S'$}, the vertex $x$ is adjacent to at least two vertices in $S'$.\\
$\bullet \; \; K_1\neq \emptyset$.  If there exists $ x\in K$ such that $|N[x]\cap S_1|=1$.\\
There exists a vertex $y\in K_2$ which is not adjacent to $s_1\in S_1$. Consider the following coloring: 
put $c(x_1)=k+1$ ($x_1$ is the vertex belonging to $K_1$ and has only $s_1$ as neighbor in $S$) and $c(y)=k+2$. 
All vertices in $B\cup K\setminus \{x_1, y\}$ receive the color $k$.\\ 
Let $x\in K$ and $s\in S$ be two vertices in $G$:\\
- \textbf{If $x\in K_1\setminus \{x_1\}$ and $s\in S'$}, then $k+1\in c(N[x])\setminus c(N[s])$.\\
- \textbf{If $x=x_1$ and $s_1$}, we have $k+2\in c(N[x])\setminus c(N[s])$.\\
- \textbf{If $x\in K_1$ and $s\in S\setminus S'$}, there exists a color $i=c(s_i)$ with 
$c(s_i)\in S'$ such that $i\in c(N[x])\setminus c(N[s])$.\\
- \textbf{If $x\in K_2$ and $s\in S'$}, the color $j\in c(N[x])\setminus c(N[s])$ with $j=c(s_j), \; \; s_j\in S'$ 
and $s_j\neq s$.\\
- \textbf{If $x\in K_2$ and $s\in S\setminus S'$}, the vertex $x$ has at least two neighbors in $S'$.
 \end{proof}\\
 
\noindent In \cite{Parreau}, the following result is given:
\begin{theorem}\cite{Parreau}
 Let $G=(K\cup S, E)$ be a split graph. \\
If $\omega (G)\geqslant 3$ or $G$ is a star, then $\chi_{lid}(G) \leq 2\omega (G)-1$.
\end{theorem}

In the following, we give two split graphs $Q_1(p)$ and $Q_2(p)$ (with $p\geq 2$) such that the lower bound of  
Theorem \ref{Split} is tight for the first, and the upper bound of Theorem \ref{split} is attained for the second graph.\\ 
The first graph $Q_1(p)$ is a split graph where  
$K=\{x_Q, Q\in \mathcal{P}(p)\}$ with $Q$ is a subset of 
$\{1, \dots, p\}$ and the stable $S=\{s_1, s_2, \dots, s_{p-1}\}$. The size of $K$ is $2^{p-1}$. 
The second graph $Q_2(p)$  is a split graph where the set-vertex of  $K$ is $\{v_1, v_2, \dots, v_{p-1}, v_p\}$ and 
the set-vertex of  $S$ is $\{s_1, s_2, \dots, s_{p-1}\}$. The edges are defined by:
\begin{itemize}
 \item  $v_is_i$ is an edge for $i=1, \dots, p-1$;
 \item  $v_iv_j$ is an edge for  $1\leq j\leq p$.
\end{itemize}

\begin{proposition}
 For $p\geq 2$, we have $\chi_{rlid}(Q_1(p))=\log(\omega (Q_1(p)))+2$. 
\end{proposition}
\begin{proof}
Let $c$ be a $rlid$-coloring of $Q_1(p)$. 
Suppose that there exist  two vertices $s_i$ and $s_j$ such that  $c(s_i)= c(s_j)$ for  
$1\leq i, \; j\leq p-1, \; \; i\neq j$ and $N[x_Q]\bigtriangleup N[x_{Q'}]=\{s_i, s_j\}$, then we obtain 
$c(N[x_Q])=c(N[x_{Q'}])$, a contradiction. Then, all $s_j$ ($1\leq i\leq p-1$) have different colors.  \\
Let $x_Q$ and $x_{Q'}$ be two vertices of $K$ such that $|Q|> |Q'|$. We have $|N[x_Q]|> |N[x_Q]|$.
In $K$, if there exists a vertex $x_Q$ having a unique neighbor $s_j\in S$ ($j=1, \dots, p-1$), then 
$N[x_Q]\bigtriangleup N[s_j]=\{x_\emptyset\}$, which implies that $c(x_\emptyset)\notin c(K\setminus \{x_\emptyset\}\cup S)$.\\
One may assume  that the  coloring $c$  of $Q_1(p)$  is defined as follows:  $c(s_i)=i$ for all $1\leq i\leq p-1$. For all 
$Q\in \mathcal{P}(p)$,  put $c(v_Q)=p$ and $c(v_\emptyset)=p+1$.\\  
Let $x_Q$ and $x_{Q'}$ be two vertices of $K$ with $|Q|>|Q'|$. We have $|c(N[x_Q])\cap S|> |c(N[x_Q])\cap S|$.
The vertex $x_\emptyset$ has not  neighbors in $S$, then it is distinguished from any vertex $x_Q$ with 
$Q\in {\mathcal P}(p)$ and $Q\neq \emptyset$. Moreover, the vertex $x_\emptyset$ is used to distinguish $s_j$ for $j=1,2$ and 
 $x_Q\in K$  with $Q\in {\mathcal P}(p)$ and $Q\neq \emptyset$. 
\end{proof}\\

For the graph $Q_2(p)$ we have:
\begin{proposition}
For $p\geq 2$, we have $\chi_{rlid}(Q_2(p))=\omega (Q_2(p))+1$. 
\end{proposition}
\begin{proof}
Let $c$ be  a $rlid$-coloring of $Q_2(p)$. 
For any vertex $s_i\in S$ with $i=1, \dots, p-1$, we have $c(s_i)\notin c(V\setminus \{s_1,  \dots, s_{p-1}\})$. Otherwise, 
we get $c(N[v_i])=c(N[v_p])$. \\
Moreover, for $1\leq i, \; j\leq  p-1$ with $i\neq j$, we have  $c(s_i)\neq c(s_j)$. Otherwise, we obtain 
$c(N[v_i])=c(N[v_j])$.\\
One may assume that the coloring $c$ of $Q_2(p)$ is defined by: for all $s_i\in S$, put   $c(s_i)=i$ with $i=1, \dots, p-1$. 
For any vertex $v_i\in K, \; \; c(v_i)\notin \{1, \dots, p-1\}$ for all $i=1, \dots, p-1$.\\
If $|c(K)|=1$ then $c(N[s_i])=c(N[v_i])$ for any $i=1, \dots, p-1$. Thus  $|c(K)|\geq 2$, which give 
$\chi_{rlid}(Q_2(p))\geq \omega (Q_2(p))+1$.\\
To conclude, observe that the coloring $c$ defined by $c(s_i)=i$ for all $i=1, \dots, p-1, \; \; c(v_i)=p$ for all $i<p$ and 
$c(v_p)=p+1$. Then $c$ is $(p+1)-rlid$-coloring of $Q_1(p)$.
\end{proof}

\section{Discussion and open problems}\label{sec7}
\noindent We have introduced the notion of \textit{relaxed locally identifying coloring} of graphs and we were interested in  
giving the minimum of colors used in a relaxed locally identifying coloring $\chi_{rlid}$. 

In Section \ref{sec1}, we have considered the graph containing twins and given a lower and an upper bound for $\chi_{rlid}$ 
for this
graphs depending on $\chi_{rlid}$ of the maximal twin-free subgraph associated to this graph.\\
In Section \ref{sec3}, we have studied the problem of complexity of $rlid$-coloring. We have shown that the problem of deciding that
$\chi_{rlid}$  equals  $3$ is $NP$-complete for $2$-degenerate planar graphs, and polynomial for bipartite graphs.\\ 
In Section \ref{sec4}, we have given a graph for which the lower bound of Theorem \ref{premiere} is tight. 
We characterized  graphs $G$  satisfying $\chi_{rlid}(G)=n$. 
We have also given a lower bound of $\chi_{rlid}$ and we exhibited a family of graphs for which 
this bound is tight. \\
In Theorem \ref{Split}, we proved that for a \textit{Split} graph $G, \; \; \chi_{rlid}(G)\leq \omega(G)+2$. 
Since we didn't find any split graph attaining this bound, we think that it will be improved.We propose the following conjecture:
\begin{conj}
 If $G$ is a free-twin split graph, then $\chi_{rlid}(G)\leq \omega(G)+1$.
\end{conj}

The notion of \textit{relaxed locally identifying chromatic number} for finite graphs is new. It might be interesting to 
investigate other graphs as cographs, planar graphs, outerplanar graphs, line graphs, interval graphs, 
for graphs with a given maximum degree.

\end{document}